\documentclass[12pt]{article}
\usepackage{amssymb,amstext,amsmath,amsthm,amscd}
\usepackage[T1]{fontenc}
\usepackage{hyperref}

\theoremstyle{definition}\newtheorem{definition}{Definition}[section]
\theoremstyle{plain}\newtheorem{theorem}[definition]{Theorem}
\theoremstyle{plain}\newtheorem{proposition}[definition]{Proposition}
\theoremstyle{plain}
\theoremstyle{plain}
\theoremstyle{definition}\newtheorem{assumption}[definition]{Assumption}
\theoremstyle{definition}\newtheorem{example}[definition]{Example}
\theoremstyle{definition}\newtheorem{remark}[definition]{Remark}

\newcommand{\range}{\mathcal{R}}

\newcommand{\N}{\mathbb{N}}
\newcommand{\cO}{\mathcal{O}}
\newcommand{\R}{\mathbb{R}}
\newcommand{\1}{{\ell^1(\N)}}
\newcommand{\2}{{\ell^2(\N)}}
\newcommand{\3}{{\ell^\infty(\N)}}
\newcommand{\diff}{\mathrm{d}\,}

\newcommand{\xad}{x_\alpha^\delta}

\newcommand{\xdag}{x^\dagger}
\newcommand{\yd}{y^\delta}

\newcommand{\la}{\langle}
\newcommand{\ra}{\rangle}

\begin{document}

\date{\today}

\title{Convergence rates in $\ell^1$-regularization when the basis is not smooth enough}

\author{
{\sc Jens Flemming}\thanks{Department of Mathematics, Technische Universit\"at Chemnitz, 09107 Chemnitz, Germany.}\,,
{\sc Markus Hegland}\thanks{Centre for Mathematics and its Applications, The Australian National University, Canberra ACT, 0200, Australia.}}

\maketitle

\begin{abstract}
Sparsity promoting regularization is an important technique for signal reconstruction and several other
ill-posed problems. Theoretical investigation typically bases on the assumption that the unknown solution
has a sparse representation with respect to a fixed basis.
We drop this sparsity assumption and provide error estimates for non-sparse solutions.
After discussing a result in this direction published earlier by one of the authors and coauthors we
prove a similar error estimate under weaker assumptions.
Two examples illustrate that this set of weaker assumptions indeed covers additional situations
which appear in applications.
\end{abstract}

\noindent\textbf{MSC2010 subject classification:} 65J20, 47A52, 49N45

\medskip


\section{Introduction}\label{s1}

Variational approaches
\begin{equation}\label{eq:tikh2}
\frac{1}{p}\|Ax-y^\delta\|_Y^p+\alpha\|x\|_\1\to\min_{x\in\1},\qquad 1 \le p < \infty,\quad \alpha>0,
\end{equation}
have become a standard tool for solving ill-posed operator equations,
\begin{equation}\label{eq:Axy}
Ax=y,
\end{equation}
for a bounded linear operator $A:X:=\1 \to Y$ mapping absolutely summable infinite sequences $x=(x_1,x_2,...)$ of real numbers $x_k,\;k \in \N$,  into a Banach space $Y$, if the solutions are known to be \emph{sparse} or if the \emph{sparsity constraints are narrowly missed}. This means that either only a finite number of nonzero components $x_k$ occurs or that the remaining nonzero components are negligibly small for large $k$. 
We assume that the exact right-hand side $y$ is in the range $\range(A):=\{Ax:\;x \in \1\}$ of $A$, which is a nonclosed subset of $Y$ due to the ill-posedness of equation (\ref{eq:Axy}), and that $y$ is not directly accessible. Instead one only has some measured noisy version $\yd\in Y$ at hand with a deterministic noise model $\|y-\yd\|_Y\leq\delta$ using the given noise level $\delta\geq 0$.  
Moreover, we assume that $\xdag \in \1$ denotes a solution of (\ref{eq:Axy}). In particular, let us suppose weak convergence
\begin{equation} \label{eq:weak}
Ae^{(k)} \rightharpoonup 0 \quad \mbox{in} \quad Y \quad \mbox{as}  \quad k \to \infty,
\end{equation}
where $e^{(k)}=(0,0,...,0,1,0,...)$
denotes the infinite unit sequence with $1$ at the $k$-th position and $0$ else. 

For successful application of $\ell^1$-regularization existence of minimizers $\xad$ to \eqref{eq:tikh2} 
and their stability with respect to perturbations in the data $\yd$ have to be ensured.
Further, by choosing the regularization parameter $\alpha>0$ in dependence on the noise level $\delta$ and
the given data $y^\delta$ one has to guarantee that corresponding minimizers converge to
a solution $\xdag$ of \eqref{eq:Axy} if the noise level goes to zero.
Such existence, stability, and convergence results can be found in the literature.
Also the verification of convergence rates has been addressed, but mostly in the case of sparse solutions (cf.~\cite{AR2,BreLor09,Daub03,Grasm09,Gras11,GrasHaltSch08,GrasHaltSch11,Lorenz08,Ramlau08,RamRes10,RamTe10,Scherzetal09}).  For non-sparse solutions a first convergence rate result can be found in \cite{BFH13}.
The aim of the present article is to formulate convergence rates results under assumptions which are weaker than those in \cite{BFH13} and
to obtain in this context further insights into the structure of $\ell^1$-regularization.

By the way we should mention that under condition (\ref{eq:weak}) in \cite{BFH13} the weak$^*$-to-weak continuity of $A$ was shown by employing the space $c_0$ of infinite sequences tending to zero, which is a predual space of $\1$, i.e., $(c_0)^*=\1$. Furthermore, we have for the range $\range(A^*)$ of the adjoint operator $A^*: Y^* \to \3$ that $\range(A^*) \subseteq c_0$ (cf.~\cite[Proposition~2.4 and Lemma~2.7]{BFH13}). These
facts, which are essentially based on (\ref{eq:weak}), ensure existence of regularized solutions $\xad$ for all $\alpha>0$ and norm convergence $\|\xad-\xdag\|_{\1} \to 0$ as $\delta \to 0$ if the
regularization parameter $\alpha=\alpha(\delta,\yd)$ is chosen in an appropriate manner, for example according to the sequential discrepancy principle (cf.~\cite{AnzHofMath12,HofMat12}). A sufficient condition to derive
(\ref{eq:weak}) is the existence of an extension of $A$ to $\2$ such that $A:\2 \to Y$ is a bounded linear operator. Then, taking into account the continuity of the embedding from $\1$ to $\2$, condition (\ref{eq:weak}) directly follows from the facts that $\{e^{(k)}\}_{k \in \N}$ is an orthogonal basis in
  $\2$ with $e^{(k)} \rightharpoonup 0$ in $\2$  as $k \to \infty$ and that every bounded linear operator is weak-to-weak continuous.  

\section{Convergence rates for smooth bases and a counter example} \label{s2}

As important ingredient and crucial condition for proving convergence rates the authors of \cite{BFH13} assumed that the following assumption holds true.

\begin{assumption}\label{as:range}
For all $k \in \N$ there exist $f^{(k)} \in Y^*$ such that
\begin{equation} \label{eq:(c)}
e^{(k)}=A^*f^{(k)}, \qquad k=1,2, ... \;.
\end{equation}
\end{assumption}

\begin{remark} \label{rem:injective}
The countable set of \emph{range conditions} (\ref{eq:(c)}) concerning the unit elements $e^{(k)}$, which form a \emph{Schauder basis} in all Banach spaces $\ell^q(\N),\;1 \le q<\infty,$ with their usual norms as well as in $c_0$ with the supremum norm,
can by using duality pairings $\langle \cdot,\cdot \rangle_{\3 \times \1}$ be equivalently rewritten as
\begin{equation} \label{eq:cont}
|\langle e^{(k)},x \rangle_{\3 \times \1}| \le C_k  \|Ax\|_Y, \qquad k=1,2,...\;, 
\end{equation} 
where, for fixed $k \in \N$, (\ref{eq:cont}) must hold for some constant $C_k>0$ and all $x \in \1$
(cf.~\cite[Lemma~8.21]{Scherzetal09}). Since we have $|x_k|=|\langle e^{(k)},x \rangle_{\3 \times \1}|$, for all $k \in \N$, Assumption~\ref{as:range} implies that $A:\1 \to Y$ is an \emph{injective} operator. 
Moreover, it can be easily verified that the following Assumption~\ref{as:smooth} is equivalent to Assumption~\ref{as:range}.
\end{remark}

\begin{assumption}\label{as:smooth}
For all $k \in \N$ there exist $f^{(k)} \in Y^*$ such that, for all $j \in \N$,
\begin{equation} \label{eq:(cH)}
\langle f^{(k)},Ae^{(j)}\rangle_{Y^* \times Y} =  \left\{\begin{array}{lcl} 1 & \mbox{if} & k=j\\0 & \mbox{if}& k \not=j \end{array}\right..
\end{equation}
\end{assumption}

The next proposition shows that the requirement (\ref{eq:(cH)}) cannot hold if one of the elements $Ae^{(j)}$ equals the sum of a convergent series $\sum _{k \in \N,\,k \not=j}\lambda_k e^{(k)}$.

\begin{proposition} \label{pro:not}
From an equation
\begin{equation}\label{eq:lambda}
\sum \limits _{k \in \N} \lambda_k Ae^{(k)}=0, \quad \mbox{where} \quad\lambda_k \in \R,\quad (\lambda_1,\lambda_2,\ldots)\neq 0,
\end{equation}
it follows that condition (\ref{eq:(cH)}) is violated.
\end{proposition}
\begin{proof}
Without loss of generality let $Ae^{(1)}=\sum \limits_{j=2}^\infty \mu_j Ae^{(j)}$ and let there exist $f^{(1)}\in Y^*$ such that (\ref{eq:(cH)}) holds. Then it follows that
$$1=\langle f^{(1)},Ae^{(1)}\rangle_{Y^* \times Y}= \sum \limits_{j=2}^\infty \mu_j\langle f^{(1)},Ae^{(j)}\rangle_{Y^* \times Y}=\sum \limits_{j=2}^\infty 0 =0,$$
which yields a contradiction and proves the proposition.\end{proof}

\begin{remark} \label{rem:smooth}
As always if range conditions occur in the context of ill-posed problems, the requirement (\ref{eq:(c)}) characterizes a specific kind of smoothness. In our case, (\ref{eq:(c)}) refers to the \emph{smoothness of the basis elements} $e^{(k)}$. Precisely, since $\range(A)$ is not a closed subset of $Y$, as a conclusion of the Closed Range Theorem (cf., e.g.,~\cite{Yos80}) we have that the range $\range(A^*)$ is also not a closed subset of $\3$ or $c_0$, and
only a sufficiently smooth basis $\{e^{(k)}\}_{k \in \N}$ can satisfy Assumption~\ref{as:range} and hence Assumption~\ref{as:smooth}. If the $\ell^1$-regularization to equation (\ref{eq:Axy}) with infinite sequences $x=(x_1,x_2,...)$ is associated to elements $Lx:=\sum \limits_{k=1}^\infty x_k u^{(k)} \in \widetilde X$ with some \emph{synthesis operator} $L:\1 \to \widetilde X$ and some Schauder basis $\{u^{(k)}\}_{k \in \N}$ in a Banach space $\widetilde X$ (see, e.g.,~\cite[Section~2]{BFH13} and \cite{Grasm09}), Assumption~\ref{as:range} refers to the smoothness of the basis elements $u^{(k)} \in \widetilde X$. The paper \cite{AnzHofRam13} illustrates this
matter by means of various linear inverse problems with practical relevance in the context of Gelfand triples. On the other hand, Example~2.6 in \cite{BFH13} indicates that operators $A$ with \emph{diagonal structure} in general satisfy Assumption~\ref{as:range}.
However, the following example will show that already for a \emph{bidiagonal structure} this is not always the case.  
\end{remark}

We give an example of an injective operator $A$ where the basis is not smooth enough to satisfy Assumption~\ref{as:smooth}, because (\ref{eq:lambda}) is fulfilled and hence by Proposition~\ref{pro:not} condition (\ref{eq:(cH)}) is violated.

\begin{example}[bidiagonal operator] \label{ex:Hegland}
{\rm For this example we consider the bounded linear operator $A:\2 \to Y:=\2$
\begin{equation}\label{eq:exA}
[Ax]_k:=\frac{x_k-x_{k+1}}{k},\quad k=1,2,...,
\end{equation}
with a bidiagonal structure. This operator is evidently injective, moreover a Hilbert-Schmidt operator due to
$$Ae^{(1)}=e^{(1)};\quad Ae^{(k)}=\frac{e^{(k)}}{k}-\frac{e^{(k-1)}}{k-1},\;k=2,3,...;$$ $$\|A\|_{HS}:=\left(\sum \limits_{k=1}^\infty \|Ae^{(k)}\|_Y^2\right)^{1/2}\le \left(2 \sum \limits_{k=1}^\infty \frac{1}{k^2} \right)^{1/2} <\infty,$$
and therefore a compact operator. Its restriction to $X:=\1$ in the sense of equation (\ref{eq:Axy}) is also injective, bounded and even compact, because the embedding operator from $\1$ to $\2$ is injective and bounded. One immediately sees that with
\begin{equation} \label{eq:Hlambda}
\sum \limits_{k=1}^\infty Ae^{(k)} =0
\end{equation}
an equation (\ref{eq:lambda}) for $\lambda_k=1,\;k\in\N$ is fulfilled.
The adjoint operator $A^*:\2 \to \2$ has the explicit representation
\begin{equation} \label{eq:Astern}
[A^\ast\eta]_1=\eta_1\qquad\text{and}\qquad[A^\ast\eta]_k=\frac{\eta_k}{k}-\frac{\eta_{k-1}}{k-1},\quad k\geq 2,
\end{equation}
and the condition (\ref{eq:(c)}) cannot hold, which also follows from the general conclusion. To satisfy, for example, the range condition
$e^{(1)}=A^*\eta$  we find successively from (\ref{eq:Astern}) $\eta_1=1,\;\eta_2=2,\;...,\;\eta_k=k,\;...$, which violates the requirement $\eta \in \2$.
}\end{example}

\section{Convergence rates if the basis is not smooth enough}

If the basis is not smooth enough, as for example when Proposition~\ref{pro:not} applies, for proving convergence rates in $\ell^1$-regularization a weaker assumption has to be established that replaces the stronger
Assumption~\ref{as:range}. We will do this in the following.

\begin{definition}[collection of source sets]\label{df:s}
For prescribed $c\in[0,1)$ we say that a sequence $\left\{S^{(n)}(c)\right\}_{n\in\N}$ of subsets
$S^{(n)}(c)\subseteq(Y^\ast)^n$ is a \emph{collection of source sets} to $c$ if, for arbitrary $n \in \N$,  $S^{(n)}(c)$   contains all elements
$(f^{(n,1)},\ldots,f^{(n,n)})\in(Y^\ast)^n$ satisfying the following conditions:
\begin{itemize}
\item[(i)]
For each $k\in\{1,\ldots,n\}$ we have $[A^\ast f^{(n,k)}]_l=0$ for $l\in\{1,\ldots,n\}\setminus k$ and
$[A^\ast f^{(n,k)}]_k=1$.
\item[(ii)]
$\sum\limits_{k=1}^n\bigl\vert[A^\ast f^{(n,k)}]_l\bigr\vert\leq c$ for all $l>n$.
\end{itemize}
\end{definition}

The properties of the $A^\ast f^{(n,k)}$ in items (i) and (ii) of the definition are visualized in
Figure~\ref{fg:f}.

\begin{figure}[h]
$$
\begin{array}{ccccccccccccc}
A^\ast f^{(n,1)}&=&(&1&0&0&\ldots&0&0&\ast&\ast&\ldots&)\\
&&&&&&&&&+&+&&\\
A^\ast f^{(n,2)}&=&(&0&1&0&\ldots&0&0&\ast&\ast&\ldots&)\\
&&&&&&&&&+&+&&\\
\vdots&&&&&&&&&\vdots&\vdots&&\\
&&&&&&&&&+&+&&\\
A^\ast f^{(n,n)}&=&(&0&0&0&\ldots&0&1&\ast&\ast&\ldots&)\\
&&&&&&&&&\leq&\leq&&\\
&&&&&&&&&c&c&&
\end{array}
$$
\caption{\label{fg:f}Structure of the vectors $(A^\ast f^{(n,1)},\ldots,A^\ast f^{(n,n)})$ in Definition~\ref{df:s}.
The sums of the absolute values of the stars in each column have to be bounded by $c$.}
\end{figure}

One easily verifies that each single source set $S^{(n)}(c)$ is convex but not necessarily closed.
It is also not clear whether the source sets are nonempty.
Thus, we will claim in the sequel the following assumption.

\begin{assumption}\label{as:source}
For some $c\in[0,1)$ there exists a collection of source sets  $\left\{S^{(n)}(c)\right\}_{n\in\N}$ which contains
only nonempty sets $S^{(n)}(c)$.
\end{assumption}

Assumption~\ref{as:source} (with $c=0$) follows from Assumption~\ref{as:range}
by observing that $(f^{(1)},\ldots,f^{(n)})\in S^{(n)}(0)$.

The construction in Definition~\ref{df:s} may look a bit technical, but the elements $f^{(n,k)}$ define
a type of approximate inverse as we will now show.

\begin{remark}
Let the linear operator $F^{(n)}:Y\rightarrow\1$ be defined as
\begin{equation}
[F^{(n)}y]_k=\begin{cases}\la f^{(n,k)},y\ra_{Y^\ast\times Y},&k=1,\ldots,n,\\0,&k=n+1,\ldots.\end{cases}
\end{equation}
The composition
\begin{equation}
Q^{(n)}:=F^{(n)}A
\end{equation}
is then a bounded linear map from $\1$ into $\1$.
The range of $Q^{(n)}$ is the set of sequences $x$ with $x_k=0$ for $k>n$.
Furthermore, one has
$$\bigl(Q^{(n)}\bigr)^2=Q^{(n)}$$
and
$$\|Q^{(n)}-P^{(n)}\|_{\1\rightarrow\1}\leq c<1,$$
where $P^{(n)}$ is the canonical projection of $\1$ onto the set of sequences $x$ with $x_k=0$ for $k>n$.
From this it follows that $Q^{(n)}$ has the infinite matrix (cf.\ Figure~\ref{fg:f})
$$\begin{pmatrix}I_n&R_n\\0&0\end{pmatrix},$$
where $I_n$ is the $n$-dimensional identity and where $\|R_n\|_{\1\rightarrow\1}\leq c$.
\end{remark}

As already noted in Remark~\ref{rem:injective}, Assumption~\ref{as:range} implies the
injectivity of the operator $A$. The subsequent proposition shows that injectivity is also necessary for fulfilling Assumption~\ref{as:source}.

\begin{proposition} \label{pro:unique}
If Assumption~\ref{as:source} is satisfied, then $A$ is injective.
\end{proposition}
\begin{proof}
Assume $Ax=0$ for some $x\in\1$. By Assumption~\ref{as:source} there exists some $c\in[0,1)$ such that for each $n\in\N$ there
is an element $(f^{(n,1)},\ldots,f^{(n,n)})\in S^{(n)}(c)$, where $\left\{S^{(n)}(c)\right\}_{n\in\N}$
denotes the collection of source sets corresponding to $c$.
For fixed $k\in\N$ and all $n\geq k$ we have
\begin{align*}
\vert x_k\vert
&=\left\vert\la A^\ast f^{(n,k)},x\ra_{\3\times\1}-\sum_{l=n+1}^\infty[A^\ast f^{(n,k)}]_lx_l\right\vert\\
&\leq\vert\la f^{(n,k)},Ax\ra_{Y^\ast\times Y}\vert+\left(\sup_{l>n}\vert[A^\ast f^{(n,k)}]_l\vert\right)\sum_{l=n+1}^\infty\vert x_l\vert\\
&\leq c\sum_{l=n+1}^\infty\vert x_l\vert.
\end{align*}
The last expression goes to zero if $n$ tends to infinity. Thus, $x_k=0$ for arbitrary $k\in\N$ and therefore $x=0$.
Note that we did not need the bound $c<1$ to prove the proposition.
\end{proof}

\begin{remark}
Assumption~\ref{as:source} can be seen as an approximate variant of Assumption~\ref{as:range}.
If $\{S^{(n)}(c)\}_{n\in\N}$ is a collection of nonempty source sets and if $(f^{(n,1)},\ldots,f^{(n,n)})\in S^{(n)}(c)$,
then $A^\ast f^{(n,k)}$ converges weakly in $\3$ to $e^{(k)}$ for each $k$.
\end{remark}

For deducing convergence rates from Assumption~\ref{as:source} we use variational inequalities, which represent an up-to-date tool
for deriving convergence rates in Banach space regularization (cf., e.g., \cite{BoHo10,Flemmingbuch12,Grasm10,HKPS07,HofYam10,SKHK12})
even if no explicit source conditions or approximate source conditions are available.
Here our focus is on convergence rates of the form
\begin{equation} \label{eq:convrate}
\|x_{\alpha(\delta,y^\delta)}^\delta-\xdag\|_\1=\cO(\varphi(\delta))\quad\text{as $\delta\to 0$}
\end{equation}
for concave rate functions $\varphi$.

{\parindent0em {\bf Condition (VIE).}}
There is a constant $\beta\in(0,1]$ and a non-decreasing, concave, and continuous function
$\varphi:[0,\infty)\rightarrow[0,\infty)$ with $\varphi(0)=0$ such that a variational inequality
\begin{equation}\label{eq:vi}
\beta\|x-\xdag\|_\1\leq\|x\|_\1-\|\xdag\|_\1+\varphi(\|Ax-A\xdag\|_Y)
\end{equation}
holds for all $x\in\1$.

\begin{theorem}\label{th:vi}
Under Assumption~\ref{as:source} Condition (VIE) is satisfied.
More precisely, we have (VIE) for $\beta=\frac{1-c}{1+c}$ with $c$ from Assumption~\ref{as:source} and for the
concave function
\begin{equation}\label{eq:phi}
\varphi(t)=2\inf_{n\in\N}\left(\sum_{k=n+1}^\infty\vert\xdag_k\vert
+\frac{t}{1+c}\inf_{f^{(n,\bullet)}\in S^{(n)}(c)}\left(\sum_{k=1}^n\|f^{(n,k)}\|_{Y^\ast}\right)\right)
\end{equation}
for $t\geq 0$, where $S^{(n)}$ is defined in Definition~\ref{df:s}. As a consequence we have the corresponding convergence rate
(\ref{eq:convrate}) for that rate function $\varphi$ when the regularization parameter is chosen according to the discrepancy principle. 
\end{theorem}

\begin{proof}
For $n\in\N$ define projections $P_n:\1\rightarrow\1$ by $[P_nx]_k:=x_k$ if $k\leq n$ and $[P_nx]_k=0$ else.
Further, set $Q_n:=I-P_n$. Then
\begin{align*}
\lefteqn{\beta\|x-\xdag\|_\1-\|x\|_\1+\|\xdag\|_\1}\\
&\qquad=\beta\|P_n(x-\xdag)\|_\1+\beta\|Q_n(x-\xdag)\|_\1-\|P_nx\|_\1-\|Q_nx\|_\1\\
&\qquad\qquad+\|P_n\xdag\|_\1+\|Q_n\xdag\|_\1.
\end{align*}
The triangle inequality yields $\|Q_n(x-\xdag)\|_\1\leq\|Q_nx\|_\1+\|Q_n\xdag\|_\1$ and
$\|P_n\xdag\|_\1\leq\|P_n(x-\xdag)\|_\1+\|P_nx\|_\1$
and therefore
\begin{align*}
\lefteqn{\beta\|x-\xdag\|_\1-\|x\|_\1+\|\xdag\|_\1}\\
&\qquad\leq(1+\beta)\|P_n(x-\xdag)\|_\1-(1-\beta)\|Q_nx\|_\1+(1+\beta)\|Q_n\xdag\|_\1\\
&\qquad=2\|Q_n\xdag\|_\1+(1+\beta)\|P_n(x-\xdag)\|_\1\\
&\qquad\qquad-(1-\beta)(\|Q_nx\|_\1+\|Q_n\xdag\|_\1).
\end{align*}
\par
Now choose $\beta=\frac{1-c}{1+c}$, which is equivalent to $c=\frac{1-\beta}{1+\beta}$,
and let $(f^{(n,1)},\ldots,f^{(n,n)})\in S^{(n)}(c)$ be arbitrary. Then
\begin{align*}
\lefteqn{\|P_n(x-\xdag)\|_\1}\\
&\qquad=\sum_{k=1}^n\left\vert\la A^\ast f^{(n,k)},x-\xdag\ra_{\3\times\1}-\sum_{l=n+1}^\infty[A^\ast f^{(n,k)}]_l(x_l-\xdag_l)\right\vert\\
&\qquad\leq\|Ax-A\xdag\|_Y\sum_{k=1}^n\|f^{(n,k)}\|_{Y^\ast}+\sum_{k=1}^n\sum_{l=n+1}^\infty\vert[A^\ast f^{(n,k)}]_l\vert\vert x_l-\xdag_l\vert
\end{align*}
and
\begin{align*}
\sum_{k=1}^n\sum_{l=n+1}^\infty\vert[A^\ast f^{(n,k)}]_l\vert\vert x_l-\xdag_l\vert
&=\sum_{l=n+1}^\infty\left(\sum_{k=1}^n\vert[A^\ast f^{(n,k)}]_l\vert\right)\vert x_l-\xdag_l\vert\\
&\leq c\|Q_n(x-\xdag)\|_\1\\
&\leq\frac{1-\beta}{1+\beta}(\|Q_nx\|_\1+\|Q_n\xdag\|_\1).
\end{align*}
Combining the estimates yields
\begin{align*}
\lefteqn{\beta\|x-\xdag\|_\1-\|x\|_\1+\|\xdag\|_\1}\\
&\qquad\leq 2\|Q_n\xdag\|_\1+(1+\beta)\|Ax-A\xdag\|_Y\sum_{k=1}^n\|f^{(n,k)}\|_{Y^\ast}\\
&\qquad=2\sum_{k=n+1}^\infty\vert\xdag_k\vert+\frac{2\|Ax-A\xdag\|_Y}{1+c}\sum_{k=1}^n\|f^{(n,k)}\|_{Y^\ast}
\end{align*}
for arbitrary $n\in\N$ and arbitrary $(f^{(n,1)},\ldots,f^{(n,n)})\in S^{(n)}(c)$.
The convergence rate result then immediately follows from Theorem~2 in \cite{HofMat12}.
\end{proof}

Now we are going to show that Theorem~\ref{th:vi} applies to the operator $A$ from Example~\ref{ex:Hegland} which does not
meet Assumption~\ref{as:range}. In this context, we see that even if a variational
inequality \eqref{eq:vi} holds for all $\beta\in(0,1)$, with possibly different functions $\varphi$, it does not
automatically hold for $\beta=1$.

\begin{proposition}
Let the operator $A$ be defined by Example~\ref{ex:Hegland} according to formula \eqref{eq:exA}. If there is a function $\varphi:[0,\infty)\rightarrow[0,\infty)$ with
$\lim \limits_{t\to 0}\varphi(t)=0$ such that Condition (VIE) with \eqref{eq:vi} holds for $\beta=1$, then we have $\xdag=0$.
\end{proposition}

\begin{proof}
Without loss of generality we assume that at least one component of $\xdag$ is positive.
For $n\in\N$ define $x^{(n)}\in\1$ by
$$x^{(n)}_k:=\begin{cases}
\xdag_k-\|\xdag\|_\3,&\text{if $k\leq n$},\\
\xdag_k,&\text{if $k>n$}
\end{cases}$$
for $k\in\N$ (if $\xdag$ has only nonpositive components, use plus instead of minus).
Then
\begin{align*}
\lefteqn{\|x^{(n)}-\xdag\|_\1-\|x^{(n)}\|_\1+\|\xdag\|_\1}\\
&=n\|\xdag\|_\3-\left(\sum_{k=1}^n\vert \xdag_k-\|\xdag\|_\3\vert
+\sum_{k=n+1}^\infty\vert \xdag_k\vert\right)+\|\xdag\|_\1\\
&=n\|\xdag\|_\3+\sum_{k=1}^n\vert \xdag_k\vert
-\sum_{\substack{k=1\\\xdag_k>0}}^n\left(\|\xdag\|_\3-\vert \xdag_k\vert\right)
-\sum_{\substack{k=1\\\xdag_k\leq 0}}^n\left(\|\xdag\|_\3+\vert \xdag_k\vert\right)\\
&=2\sum_{\substack{k=1\\\xdag_k>0}}^n\vert \xdag_k\vert
\end{align*}
and
$$\|Ax^{(n)}-A\xdag\|_Y=\frac{1}{n}\|\xdag\|_\3.$$
Thus, the variational inequality \eqref{eq:vi} implies
$$2\sum_{\substack{k=1\\\xdag_k>0}}^n\vert \xdag_k\vert\leq\frac{1}{n}\|\xdag\|_\3$$
for all $n\in\N$, which is a contradiction since the left-hand side is bounded away from zero but the right-hand
side approaches zero for large $n$.
\end{proof}

The next proposition together with Theorem~\ref{th:vi} shows that for each $\beta\in(0,1)$ a
variational inequality is valid and hence a corresponding convergence rate (\ref{eq:convrate}) can be established for $A$ from Example~\ref{ex:Hegland},
where the rate function $\varphi$ arises from properties of $A$ in combination with the decay rate of $|\xdag_k| \to 0$ as $k \to \infty$.

\begin{proposition}
For the operator $A$ from Example~\ref{ex:Hegland} the Assumption~\ref{as:source} holds for all $c\in(0,1)$.
\end{proposition}

\begin{proof}
Let $c\in(0,1)$ and set
$$a:=\left\lfloor\frac{1}{c}\right\rfloor\qquad\text{and}\qquad b:=1-ca.$$
Then $a\in\N$, $b\in[0,c)$, and $ca+b=1$.
For $n\in\N$ and $k\in\{1,\ldots,n\}$ define $e^{(n,k)}\in\3$ by
$$e^{(n,k)}_{ln+p}:=\begin{cases}
1,&\text{if $l=0$, $p=k$},\\
-c,&\text{if $l\in\{1,\ldots,a\}$, $p=k$},\\
-b,&\text{if $l=a+1$, $p=k$},\\
0,&\text{else}\\
\end{cases}$$
for $l\in\N_{0}$ and $p\in\{1,\ldots,n\}$.
Then $\sum_{l=1}^Ne^{(n,k)}_l=0$ for all $N>(a+2)n$.
Thus, the sequence $f^{(n,k)}$ defined by
$$f^{(n,k)}_l:=l\sum_{m=1}^le^{(n,k)}_m,\quad l\in\N,$$
belongs to $Y^\ast=\2$ and we have $e^{(n,k)}=A^\ast f^{(n,k)}$.
Item (i) in Definition~\ref{df:s} is obviously satisfied and item (ii)
can be easily deduced from the fact that for fixed
$n\in\N$ the elements $e^{(n,1)},\ldots,e^{(n,n)}$ have mutually disjoint supports.
Since for each $n\in\N$ we found $(f^{(n,1)},\ldots,f^{(n,n)})\in S^{(n)}$, the source sets
$S^{(n)}$ are nonempty.
\end{proof}

\section{Another example: integration operator and Haar wavelets}\label{sc:example}

In addition to Example~\ref{ex:Hegland} we now provide another, less artificial, example for a situation
where Assumption~\ref{as:source} is satisfied but Assumption~\ref{as:range} is violated.

For $\tilde{X}:=L^2(0,1)$ and $Y:=L^2(0,1)$ define $\tilde{A}:\tilde{X}\rightarrow Y$ by
\begin{equation}
(\tilde{A}\tilde{x})(s):=\int_0^s\tilde{x}(t)\diff t,\quad s\in(0,1).
\end{equation}
Then $\tilde{X}^\ast=L^2(0,1)$ and $Y^\ast=L^2(0,1)$, too, and
$\tilde{A}^\ast:Y^\ast\rightarrow\tilde{X}^\ast$ is given by
\begin{equation}
(\tilde{A}^\ast y)(t)=\int_{t}^1y(s)\diff s,\quad t\in(0,1).
\end{equation}

Suppose we know that the unknown solution to $\tilde{A}\tilde{x}=y$ is sparse or at least nearly sparse with respect to the Haar
basis and denote the synthesis operator of the Haar system by $L:\2\rightarrow L^2(0,1)$.
Then for given noisy data $y^\delta$ we would like to find approximate solutions $\tilde{x}_\alpha^\delta:=Lx_\alpha^\delta$,
where $x_\alpha^\delta\in\1$ is the minimizer of \eqref{eq:tikh2} with $A:=\tilde{A}L$.

To obtain convergence rates for that method we have to verify Assumption~\ref{as:range} or Assumption~\ref{as:source}.
But first let us recall the definition of the Haar basis.

The first element of the Haar system is given by $u^{(1)}(s):=1$ for $s\in(0,1)$. All other elements are scaled and
translated versions of the function
$$\psi(s):=\begin{cases}1,&s\in(0,\frac{1}{2}),\\-1,&s\in(\frac{1}{2},1).\end{cases}$$
More precisely
$$u^{(1+2^l+k)}(s):=\psi_{l,k}(s):=2^{\frac{l}{2}}\psi(2^ls-k),\quad s\in(0,1),$$
for $l=0,1,2\ldots$ and $k=0,1,\ldots,2^l-1$.

The following proposition shows that Assumption~\ref{as:range} is not satisfied in this setting, that is,
the basis $\{e^{(k)}\}_{k\in\N}$ in $\1$, and thus the Haar basis in $L^2(0,1)$, is not smooth enough with respect
to $A$ to obtain convergence rates via Assumption~\ref{as:range}.

\begin{proposition}
The element $e^{(1)}$ does not belong to $\range(A^\ast)$ and, thus, Assumption~\ref{as:range} does not hold.
\end{proposition}

\begin{proof}
Assume that there is some $f^{(1)}\in Y^\ast$ such that $e^{(1)}=A^\ast f^{(1)}$.
Then $1=(\tilde{A}^\ast f^{(1)})(s)$ for all $s$, but elements from $\range(\tilde{A}^\ast)$
always belong to the Sobolev space $H^1(0,1)$ and are zero at $s=1$.
Thus, $1=(\tilde{A}^\ast f^{(1)})(1)$ cannot be true and $e^{(1)}=A^\ast f^{(1)}$ is not possible.
\end{proof}

The second proposition states that Assumption~\ref{as:source} is satisfied and thus convergence rates
can be obtained for our example.

\begin{proposition}
There is a nonempty collection of source sets $S^{(n)}(c)$ with $c=(4-\sqrt{8})^{-1}<1$.
For $n=1$ the element $f^{(1,1)}:=2$ belongs to $S^{(n)}(c)$.
For $n=2^m$ with $m\in\N$ the vector $(f^{(2^m,1)},\ldots,f^{(2^m,2^m)})$ given by
\begin{equation}
f^{(2^m,1+q)}:=-2^{\frac{m}{2}+2}\sum_{p=0}^{2^m-1}c^{(m)}_{1+q,1+p}\psi_{m,p},\quad q=0,\ldots,2^m-1,
\end{equation}
belongs to $S^{(n)}(c)$.
Here,
$$c^{(m)}_{1,1+p}=1$$
and
$$c^{(m)}_{1+2^r+s,1+p}:=\begin{cases}2^{\frac{r}{2}},&2^{m-r}s\leq p\leq 2^{m-r}(s+\frac{1}{2})-1,\\
-2^{\frac{r}{2}},&2^{m-r}(s+\frac{1}{2})\leq p\leq 2^{m-r}(s+1)-1\\
0,&\text{else},\end{cases}$$
for $r=0,\ldots,m-1$ and $s=0,\ldots,2^r-1$ and $p=0,\ldots,2^m-1$.
For all other values of $n$ an element contained in $S^{(n)}(c)$ can be obtained by truncating a vector
from $S^{(2^m)}(c)$ with $2^m\geq n$ after $n$ components.
\par
The sum in \eqref{eq:phi} can be estimated for $n\leq 2^m$ by
\begin{equation}
\sum_{k=1}^n\|f^{(n,k)}\|_{Y^\ast}\leq 2+2^{2m+2}-2^{m+2}
\end{equation}
where equality holds if $n=2^m$.
\end{proposition}

\begin{proof}
The proposition can be proven by elementary but lengthy calculations.
We only provide the elements $A^\ast f^{(n,k)}$ and the estimate for $c$.
\par
For $n=1$ we have
$$[A^\ast f^{(1,1)}]_1=1\quad\text{and}\quad
[A^\ast f^{(1,1)}]_{1+2^l+k}=2^{-\frac{3}{2}l-1},$$
where $l\in\N_0$ and $k=0,\ldots,2^l-1$.
\par
For $n=2^m$ with $m\in\N$ the element $A^\ast f^{(2^m,1+q)}$ has zeros in the first $2^{m+1}$ components
except for position $q$, where a one appears.
The absolute values of the remaining components are given by
$$\bigl\vert[A^\ast f^{(2^m,1)}]_{1+2^l+k}\bigr\vert=2^{m-\frac{3}{2}l}$$
and
$$\bigl\vert[A^\ast f^{(2^m,1+2^r+s)}]_{1+2^l+k}\bigr\vert
=\begin{cases}2^{\frac{1}{2}r+m-\frac{3}{2}l},&2^{-(l-r)}(k+1)-1\leq s\leq 2^{-(l-r)}\leq k,\\
0,&\text{else},\end{cases}$$
where $l>m$, $k=0,1,\ldots,2^l-1$, $r=0,1,\ldots,m-1$, $s=0,1,\ldots,2^r-1$.
\par
Now we come to the estimate of the constant $c$.
First, note that for fixed $m$, $r$, $l$, $k$ there is at most one $s$ such that $[A^\ast f^{(2^m,1+2^r+s)}]_{1+2^l+k}\neq 0$.
Then for fixed $m$, $l$, $k$ with $l>m$ we have
\begin{align*}
\sum_{\kappa=0}^{2^m-1}\bigl\vert[A^\ast f^{(2^m,\kappa)}]_{1+2^l+k}\bigr\vert
&=\bigl\vert[A^\ast f^{(2^m,1)}]_{1+2^l+k}\bigr\vert+\sum_{r=0}^{m-1}\sum_{s=0}^{2^r-1}\bigl\vert[A^\ast f^{(2^m,1+2^r+s)}]_{1+2^l+k}\bigr\vert\\
&\leq 2^{m-\frac{3}{2}l}+\sum_{r=0}^{m-1}2^{\frac{1}{2}r+m-\frac{3}{2}l}
=2^{m-\frac{3}{2}l}\left(1+\sum_{r=0}^{m-1}\bigl(\sqrt{2}\bigr)^r\right)\\
&=2^{m-\frac{3}{2}l}\left(1+\frac{\bigl(\sqrt{2}\bigr)^m-1}{\sqrt{2}-1}\right).
\end{align*}
Using $l\geq m+1$ we further estimate
\begin{align*}
\sum_{\kappa=0}^{2^m-1}\bigl\vert[A^\ast f^{(2^m,\kappa)}]_{1+2^l+k}\bigr\vert
&\leq 2^{-\frac{1}{2}m-\frac{3}{2}}\left(1+\frac{\bigl(\sqrt{2}\bigr)^m-1}{\sqrt{2}-1}\right)\\
&=2^{-\frac{3}{2}}\left(2^{-\frac{1}{2}m}+\frac{1-2^{-\frac{1}{2}m}}{\sqrt{2}-1}\right)\\
&=\frac{2^{-\frac{3}{2}}}{\sqrt{2}-1}\left(1-\bigl(2-\sqrt{2}\bigr)2^{-\frac{1}{2}m}\right)\\
&\leq\frac{2^{-\frac{3}{2}}}{\sqrt{2}-1}
=\frac{1}{4-\sqrt{8}}.
\end{align*}
\end{proof}

\section{Conclusions}

We have shown that the source condition $e^{(k)}\in\range(A^\ast)$ for all $k$, as published in \cite{BFH13}
for obtaining convergence rates in $\ell^1$-regularization, is rather strong.
A weaker assumption yields a comparable rate result and covers a wider field of settings.
Especially nonsmooth bases, e.g.\ the Haar basis, can be handled even if the forward operator is smoothing
and the basis elements do not belong to the range of the adjoint.

A major drawback of our new approach (and of the previous one in \cite{BFH13}) is that
the assumptions automatically imply injectivity of the operator.
Sufficient condition for convergence rates in $\ell^1$-regularization if the operator is not
injective and if the solution is not sparse are not known up to now. 

\section*{Acknowledgments}

The authors thank Bernd Hofmann for many valuable comments on a draft of this article and for fruitful discussions on the subject.
J.\ Flemming was supported by the German Science Foundation (DFG) under grant FL~832/1-1.
M.\ Hegland was partially supported by the Technische Universit\"at M\"unchen Institute of Advanced Study,
funded by the German Excellence Initiative.
Work on this article was partially conducted during a stay of M.\ Hegland at TU Chemnitz, supported by the German Science
Foundation (DFG) under grant HO~1454/8-1.

\end{document}